\documentclass[12pt,a4paper]{article}
\usepackage[cp1251]{inputenc}
\usepackage[english]{babel}
\usepackage{amsmath,amssymb,amsopn,amsthm}
 \usepackage{amscd,amssymb,epsfig, epsf}
\usepackage{graphicx}

\theoremstyle{plane}
\newtheorem{theorem}{Theorem}
\newtheorem{lemma}{Lemma}
\newtheorem{definition}{Definition}

\newtheorem{corollary}{Corollary}

\def\p{\partial}

\def\Int{{\mbox{Int}}}
\def\endproof{\hfill \vrule height8pt width4pt}
\begin{document}

\begin{center} {\bf \Large A semigroup of theta-curves
 in 3-manifolds} \end{center}

  \begin{center} {\rm  \sc Sergei Matveev\footnote{Partially supported
 by the RFBR grant
08-01-162   and the Program of Basic Research of RAS, project
09-T-1-1004.}, Vladimir Turaev\footnote{Partially supported
 by the NSF grant DMS-0904262.}}\end{center}

{\sc Abstract.} {\small We establish an existence and uniqueness
theorem for  prime decompositions of theta-curves   in
$3$-manifolds.}

 \section {Introduction }

 This paper is concerned with the existence and uniqueness of prime
decompositions of theta-graphs   in $3$-manifolds. A theta-graph is
a graph formed by two ordered vertices,  called the {\em leg} and
the {\em head}, and three oriented edges leading from the leg to the
head and labeled with the symbols $\{-,+,0\}$  (different edges
should have different labels.) Theta-graphs embedded in
$3$-manifolds are called theta-curves; their study is parallel   to
the  study of knots (i.e., knotted circles) in $3$-manifolds.  More
precisely, a {\em theta-curve} is a pair $(M,\Theta)$, where $M$ is
a compact connected oriented $3$-manifold and $\Theta   $ is a
theta-graph embedded in $\Int M$. By abuse of language, we will call
$\Theta$ a theta-curve in~$M$. For example, every
 $3$-manifold $M$ contains a unique (up to isotopy)     {\em flat theta-curve}       that
  lies
in a 2-disc embedded into $M$. The flat theta-curve in $S^3$  is
called the {\em trivial theta-curve}. Further examples of
theta-curves   can be obtained by tying knots on the edges of   flat
theta-curves (see below for   details). The resulting theta-curves
are said to be {\em  {knot-like}}.

By homeomorphisms of theta-curves we mean homeomorphisms of pairs
preserving orientation in the ambient $3$-manifolds and the
orientation and the labels  of the  edges of theta-curves. The set
of homeomorphism classes of theta-curves is denoted ${\cal T} $. We
define a {\it vertex multiplication} in ${\cal T} $, see \cite{wo}
for the case of theta-curves in $ S^3$.   Given theta-curves $
(M_i,\Theta_i)$ with $ i=1,2$,
    pick    regular neighborhoods $B_1 \subset M_1$ and $B_2\subset M_2 $ of
    the head of $\Theta_1 $ and
      the leg of $\Theta_2$, respectively.
Glue     $M_1\setminus \Int\ B_1$ and $M_2\setminus \Int \ B_2$
along an orientation-reversing homeomorphism $
\partial B_1\to
\partial B_2$   that  carries  the only intersection point of $
\partial B_1$ with the $i$-labeled  edge of $\Theta_1 $ to the  intersection
point of $  \partial B_2$ with  the $i$-labeled edge of $\Theta_2$
for  $i\in \{ -, 0, + \}$.
 The union $\Theta$ of  $\Theta_1 \cap (M_1\setminus \Int\ B_1)$ and   $\Theta_2 \cap (M_2\setminus \Int\ B_2)$ is a
theta-curve   in $ M=M_1\# M_2$. The theta-curve $ (M,\Theta)$  is
called the vertex product of $\theta_1=(M_1,\Theta_1)$,
$\theta_2=(M_2,\Theta_2)$ and denoted $\theta_1 \circ\theta_2$.

The vertex multiplication is associative and    turns ${\cal T} $
into a semigroup. The unit of ${\cal T} $ is the  trivial
theta-curve. The semigroup ${\cal T}$ is non-commutative but has a
big center: it follows from the definitions that all knot-like
theta-curves lie in the center of ${\cal T}$. Note also that
  a product theta-curve
$\theta_1\circ \theta_2$ is  trivial  if and only if
 both $\theta_1$ and $\theta_2$ are trivial, see \cite{Mo1,wo}.

We call a theta-curve {\em  prime} if it is non-trivial and does not
expand as a product of two non-trivial theta-curves. The following
theorem is the main result of this paper.

\begin{theorem} \label{main} Let $\theta=(M,\Theta)$ be a non-trivial theta-curve such
that all $2$-spheres in $M$ are separating. Then:

\begin{enumerate}
\item $\theta$ expands as a   product   $\theta=
\theta_1\circ \theta_2\circ \dots \circ \theta_n$  for a finite
 sequence $\theta_1, \ldots ,  \theta_n$ of prime theta-curves.

\item This expansion is unique up to relations of  type
$\theta'\circ \theta''=\theta''\circ \theta'$, where   $\theta'$
or $ \theta''$ is   knot-like.
\end{enumerate}

\end{theorem}

For   $M=S^3$, this theorem is due to Motohashi \cite{Mo1}. A
similar   theorem for   knots in $3$-manifolds is also true; we
obtain it at the end of the paper as a corollary of
Theorem~\ref{main}.

     A study of prime decompositions is a traditional area of
$3$-dimensional topology. We refer to \cite{Milnor} and
\cite{Shubert} for prime decompositions of $3$-manifolds and   knots
in~$S^3$, to~\cite{miya} for prime decompositions of   knots in
$3$-manifolds, and to \cite{Petronio}, \cite{Mat-HA} for prime
decompositions of orbifolds and knotted graphs in $3$-manifolds.
Our interest  in prime decompositions of theta-curves is due to a connection
to so-called knotoids recently introduced by the second named author \cite{tu2}.

Let us explain the reasons for our assumption on the $2$-spheres in
Theorem~\ref{main}. If $M$ contains  a non-separating $2$-sphere
$S$, then there is a theta-curve $\Theta\subset M$ meeting $S $
transversely in  one point of  an edge   of $\Theta$. Using $S$ it
is easy to see that tying any local knot on this edge one obtains a
theta-curve isotopic to $\Theta$.  This yields an infinite family of
knot-like theta-curves that are   factors of $(M,\Theta)$ and
compromises  the existence and uniqueness  of prime decompositions
of $(M,\Theta)$.

To prove Theorem~\ref{main} we follow the general scheme introduced
in \cite{Mat-HA}. This scheme has been successfully applied to prove
the existence and uniqueness  of prime decompositions in many
similar geometric situations.

\section{Preliminaries on knots}\label{knots}

\begin{definition} \label{df:knot} A knot is a pair
$(Q,K)$ where $Q$ is a compact connected oriented $3$-manifold
 and $K$  is an oriented
simple closed curve in $\Int\  Q$. Two knots $(Q,K), (Q', K')$ are {
equivalent}, if there is a homeomorphism $(Q,K)\to (Q',K')$
preserving orientations of both $Q$ and $K$.
\end{definition}

  We call a knot $ (Q,K)$ {\em
flat} if $K$ bounds an embedded disc in $Q$. A knot $ (Q,K)$ is
 {\em trivial} if $Q=S^3$  and $K$ is   flat. We
emphasize that all  knots in $Q\neq S^3$ are non-trivial. Denote by
${\cal K} $ the set  of all equivalence classes of knots. We equip
${\cal K} $ with a binary operation $\#$ (connected sum) as follows.
Let $k_i=(Q_i,K_i) \in {\cal K}  $ for $i=1,2$. Choose a closed
$3$-ball $B_i\subset Q_i $ such that $\l_i =B_i\cap K_i $ is an
unknotted arc in $B_i$. Let $h\colon (B_1,l_1) \to (B_2,l_2)$ be a
homeomorphism which reverses orientations of both the ball and the
arc. Glue $Q_1\setminus \Int \  B_1$ and $Q_2\setminus \Int \ B_2$
 along  $h \vert_{\p B_1}\colon (\p B_1,\p l_1)
\to (\p B_2,\p l_2)$. The resulting knot $ (Q_1\# Q_2 ,K_1\# K_2 )$
does not depend on the choice of $B_1, B_2$, and $h$. This knot  is
called the {connected sum } of $k_1, k_2$ and denoted $k_1\# k_2$.

The operation $\#$ is commutative, associative, and has    a neutral
element represented by the trivial knot. A classical argument due to
Fox \cite{F}  shows that the knot $k=k_1\# k_2$ is trivial   if and
only if both $k_1$ and $k_2$ are trivial. Namely, if $k$ is trivial,
then $\#_{i=1}^\infty k_i=k$ and $k_1\#(\#_{i=2}^\infty k_i)=k_1$.
Therefore $k_1=k$ is trivial.

 \section{From knots to theta-curves}\label{knots+}

Given a knot $k=(Q,K)$   and a label $i\in \{-,0,+\} $, we  define a
theta-curve in $Q$ as follows. Pick a disc $D\subset Q$ meeting $K$
along an arc $l'=D\cap K=\partial D\cap K$.    The complementary arc
$l=K\setminus \Int \ l'$ of $K$ receives the label $i$ and the
orientation induced by that of $K$, the arcs $l'$ and $l''=\overline
{\partial D \setminus l'}$ receive the remaining labels. Then
$\Theta_K=l\cup l'\cup l''$ is a theta-curve in $Q$ (cf.\
Figure~\ref{tau}). The homeomorphism class of the theta-curve $(Q,
\Theta_K)$ does not depend on the choice of $D$ because any two such
disks are isotopic. This class is denoted $\tau_i (k)$.

It is easy to see from the definitions that the map $  {\cal K} \to
{\cal T}, k\mapsto \tau_i(k)$ is a semigroup homomorphism. This
homomorphism is injective because its composition with the map
${\cal T} \to \cal K$ removing the $j$-labeled edge (for $j\neq i$)
is the identity.

A theta-curve is {\em knot-like} if it lies in the image of one of
$\tau_i$ for $i\in \{-,0,+\} $. As was   mentioned above, the
knot-like theta-curves commute with all theta-curves, i.e., lie in
the center of ${\cal T}$.

Given a knot $k=(Q,K)$, a theta-curve
  $\theta=(M, \Theta)$, and a label $i\in \{-,0,+\} $, we   define
  the {\it knot
insertion} of $k$ into $\theta$ to be the theta-curve $ \tau_i (k)
\circ \theta=\theta \circ \tau_i (k)$.  To construct this
theta-curve geometrically,  pick a $3$-ball $B \subset M$   such
that $B\cap \Theta$  is an unknotted arc in the $i$-labeled edge of
$\theta$. The theta-curve $ \tau_i (k) \circ \theta=\theta \circ
\tau_i (k)$ is obtained by cutting off $(B, B\cap \Theta)$ from
$(M,\Theta)$ and coherent filling the resulting hole by $(Q,K)$. For
$Q=S^3$, this is the standard tying of local knots on the edges of
$\theta$.

\begin{figure}
\centerline{\psfig{figure=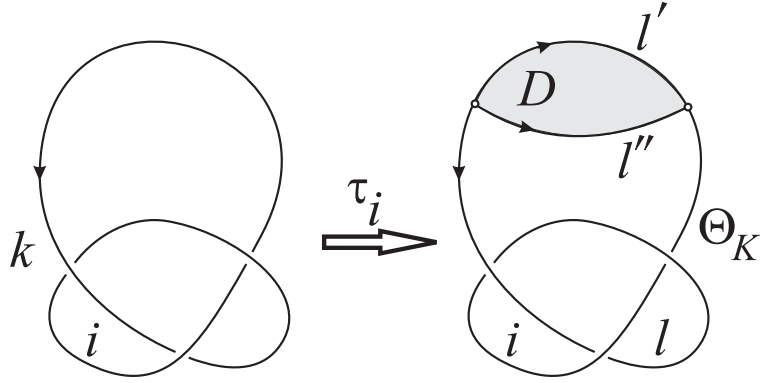,height=5cm}}
  \caption{The map $\tau_i$}
  \label{tau}
\end{figure}

In analogy with $\tau_i$, we define a homomorphism $\tau \colon
{\cal M} \to {\cal T}$, where ${\cal M} $ is the semigroup of
compact connected oriented $3$-manifolds
 with respect to connected
summation. If $M\in \cal M$, then $\tau (M)$ is $M$ with a flat
theta-curve inside. This suggests a notion of a {\em manifold
insertion}.   The insertion of $M\in {\cal M}$ into a theta-curve
$\theta=(Q, \Theta)$ yields the theta-curve $\tau(M)\circ
\theta=\theta \circ \tau (M)$  obtained by replacing a ball in $
Q\setminus \Theta$ by a copy of punctured $M$. The same theta-curve
can be obtained by inserting a  flat knot in $M$ into $  \theta$.

\section{Prime theta-curves and knots}\label{tc3}

\begin{lemma}\label{le1} Let $k=(Q,K)$ be a knot, $i\in \{-,0,+\}$,  and
$\tau_i(k)=(Q,\Theta_K)$   the corresponding theta-curve. Let $D$ be
a disc in $Q$ such that $\p D $ is the union of two edges of
$\Theta_K$ with labels distinct from  $i$. If $S\subset Q$ is a
$2$-sphere
 meeting each edge of $\Theta_K$ in one point, then there
is a self-homeomorphism of $Q$ which keeps $\Theta_K$ fixed and
carries $S$ to a $2$-sphere $S' $ such that $S'\cap D$ is a single
arc.
\end{lemma}

\begin{proof} The set   $S\cap D$ consists of an arc $\alpha$ joining two
 points of $S\cap \Theta_K$ and possibly of several
circles. The   innermost circle argument yields a disc $A\subset S$
such that $A\cap D=\p A$. The circle  $\p A$ bounds a disc
$A'\subset D$.
 Then the disc
$(D\setminus A')\cup A$ is isotopic to a disc  in $Q$ which spans
the same edges of $\Theta_K$ and crosses $S$ along $\alpha$ and
fewer
  circles. Continuing by induction, we obtain a spanning disc
$D'$ such that $D'\cap S=\alpha$. There is a homeomorphism $h\colon
Q\to Q$ that keeps $\Theta_K$ pointwise and carries $D'$ to $D$.
Then $S'=h(S)$ is a required sphere.
\end{proof}

\begin{lemma}\label{le2} A knot $k=(Q,K)$   is prime if and only if
the   theta-curve ${\tau}_i(k)=(Q,\Theta_K)$ is prime for some (and
hence for any)  $i\in\{ -,0,+\}$. The 3-manifold $Q$ is   prime
  if and only if the flat theta-curve $ \tau(Q)$ is prime.
\end{lemma}

\begin{proof} Recall that   a knot (resp., a theta-curve)   is prime if it
is non-trivial and does not split as a connected  sum (resp., a
product) of two non-trivial knots (resp.,   theta-curves). Since
$\tau_i$ is an injective homomorphism, if ${\tau}_i(k)$ is prime
then so is $k$. Suppose that that ${\tau}_i(k)$ is not prime. Then
there is a sphere $S\subset Q$ meeting each edge of $\Theta_K $ in
one point and dividing $(Q,\Theta_K)$ into two   pieces $(Q_j,
Q_j\cap \Theta_K), j=1,2$,
  not homeomorphic to a $3$-ball with three radii. Denote by
$D$ a disc spanning the edges of $\Theta_K$ with labels distinct
from $ i$. By Lemma~\ref{le1} we may assume that $S \cap D$ is an
arc dividing $D$ into two subdiscs.  We conclude that after deleting
one of the edges spanned by $D$, i.e., after returning to $k=(Q,K)$,
the pieces $(Q_j, Q_j\cap K)$ remain non-trivial, i.e., are not
homeomorphic to a $3$-ball with two radii. We   conclude that $k$
splits as a connected  sum of two non-trivial knots.

The second claim of the lemma is obtained by applying the first
 claim to the flat knot $K_0\subset Q$. It is clear that $Q$ is prime
 if and only if the knot $(Q,K_0)$ is prime. The latter holds if and
 only if the theta-curve $ \tau_i(K_0)=\tau(Q)$ is prime.
\end{proof}

\section{Spherical reductions}\label{spr} Spherical reductions are
operations inverse to taking products of theta-curves and
inserting
  knots. Denote by $\cal U$ the set of all pairs $(M,G)$, where $M$ is a
 compact connected oriented  $3$-manifold  and $G\subset M$ is either a theta-graph, or a knot labeled
 by $i\in \{-,0,+\}$,
  or the empty set. The pairs are considered up to homeomorphisms preserving all
  orientations and labels.
  In other words, ${\cal U}={\cal T} \bigsqcup   {\cal K}_- \bigsqcup   {\cal K}_0 \bigsqcup   {\cal K}_+ \bigsqcup {\cal
  M}$,  where
   ${\cal K}_i$ is the set of
 $i$-labeled knots.
\begin{definition} Let $(M,G)\in \cal U$.  A  separating sphere
 $S$   in $M$ is {  admissible}
if it is in general position with respect to $G$  and $S\cap G$ is
either empty or consists of $2$ or $3$ points.
\end{definition}

\begin{definition} \label{reductions} Given  an admissible sphere $S$  in
$(M,G)\in \cal U$, we cut $(M,G)$ along $S$ and add cones over two
copies of $(S, S\cap G)$  on the boundaries of the resulting two
pieces of $(M,G)$.  This gives two   pairs $(M_j,G_j)\in {\cal U},
j=1,2$, where the orientations and labels of the edges of $G_j$ are
inherited from those of $G$.
 We say that these pairs    are
  obtained by { spherical reduction} of $(M,G)$ along $S$.
\end{definition}

The sphere $S$ as above and the reduction  along $S$  are {\it
inessential} if $S$
  bounds a $3$-ball $B\subset M$ such that $B\cap G$ is either empty, or a proper unknotted arc, or
  consists of three radii of $B$.  The reduction along an inessential sphere produces a copy of
$(M,G)$ and a trivial pair, which is either    $ (S^3, \emptyset) $,
or a trivial knot, or a trivial theta-curve.

\section{Mediator spheres}
\begin{definition}\label{mediator} Let $M$ be a $3$-manifold and
$S_1,S_2,S_3$ three mutually transversal    spheres in $M$. We call
  $S_3 $   a sphere-mediator for
$S_1,S_2$, if both numbers $\#(S_3\cap S_2)$ and  $\#(S_3\cap S_1)$
are strictly smaller than $\#(S_2\cap S_1)$. Here $\#$ denotes the
number of circles.
\end{definition}

\begin{lemma} \label{clean} Let $(M,G)\in {\cal U}$ and
$S_1,S_2$ be     admissible  essential mutually transversal spheres
in $M$ such that $S_1\cap S_2\neq \emptyset$. If   all $2$-spheres
in $M$ are separating, then there is an essential sphere-mediator
$S_3 $ for $S_1,S_2$.
\end{lemma}

\begin{proof} Using an innermost circle argument, we can find   two
 discs in    $ S_1$ intersecting  $S_2$ solely along their
boundaries.   Since   $S_1$ meets $G$ in $\leq 3$ points, one of the
discs,   $D$,  meets $G$ in $\leq 1$ point. The circle $\p D $
splits $S_2$ into two discs $D',D''$ such that $S'=D'\cup D$ and
$S''=D''\cup D$ are embedded spheres in~$M$. Since all spheres in
$M$ are separating,  $S' $ and $S''$  bound    manifolds
$W',W''\subset M$ respectively so that $W'\cap W''=D$ and $\p
(W'\cup W'')=S_2$. Let $X$ be the closure of $M\setminus (W'\cup
W'')$, see Figure~\ref{s1s2}.

\begin{figure}
\centerline{\psfig{figure=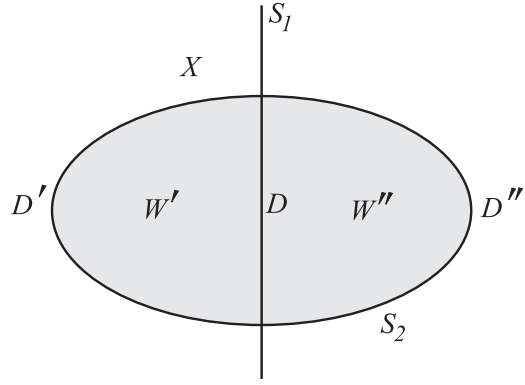,height=5cm}}
  \caption{The spheres $S_1$ and $S_2$}
  \label{s1s2}
\end{figure}

Case 1:   $D\cap G=\emptyset$. Since the intersection of $G$ with a
separating sphere cannot consist of one point and $G$ meets $S_2$ in
$\leq 3$ points, at least one of the spheres $S'$, $S''$, say, $S'$,
does not meet $G$.   If $S'=\p W'$ is essential, then pushing it
slightly inside $ W'$ we obtain a sphere-mediator for $S_1,S_2$.
Indeed, the latter sphere is disjoint from $S_2$ and meets $S_1$ in
fewer circles.

If $S'$ is inessential, then it bounds a $3$-ball in $M$ disjoint
from $G$. This ball is
 either $W'$ or   $W''\cup X$. The second option is impossible, since
 $W''\cup X$ contains  the essential sphere
  $S_2$. Hence $W'$   is a $3$-ball, and we can use
  it  to isotope $D'$ to the other side of $S_1$ and thus transform  $S_2$ into a parallel copy of $S''$
 disjoint from $S_2$.
 This copy of $S''$ is a sphere-mediator for $S_1,S_2$:   it intersects $S_1$ in fewer circles than $S_2$
 and is essential, since     $S_2$ is essential.

Case 2:   $D\cap G $ is a one-point set. An  argument  as above
shows that
    one of the
spheres  $S',S''$, say
 $S'$, meets $G$ in two points.
  If $S'$ is essential, then after a small isotopy it can be taken as a
  sphere-mediator.
If $S'$ is inessential, then $W'$  is a $3$-ball and  $W' \cap G$ is
an unknotted arc. As above,
   we can      use $W'$ to isotope $D'$ to the other side of $S_1$ and thus transform  $S_2$ into a
  sphere-mediator for $S_1$, $S_2$.
\end{proof}

\section{  Digression into theory of roots} Let $\Gamma$ be an oriented graph.  The set of vertices of   $\Gamma$ will be denoted $\mathbb{V}(\Gamma)$.
 By a {\emph path} in $ \Gamma $ from a vertex $V$ to a vertex $W$ we mean a sequence of coherently
 oriented edges $\overrightarrow{VV_1},\overrightarrow{V_1V_2}, \dots, \overrightarrow{V_nW}$, where $V_1, ..., V_n\in \mathbb{V}(\Gamma)$.
 A vertex
 $W$ of $\Gamma$    is a {\em subordinate} of a vertex $V$, if  either $V=W$ or there
 is a  path from $V$ to
 $W$ in $\Gamma$.
   A vertex
 $W$   is  a {\em root} of   $V$, if $W$ is a subordinate of  $V$ and $W$ has no outgoing edges.

 We say that $\Gamma$ has
 property {\rm (F)} if for any vertex  $V\in \mathbb{V}(\Gamma)$ there is an integer $C\geq 0 $ such that
  any path in $ \Gamma $ starting at $V$
consists of no more than $C$ edges.
  It is obvious that if $\Gamma$ has
 property {\rm (F)} then every vertex of $\Gamma$ has a
  root. To study the uniqueness of the root, we need the following
  notion.

\begin{definition}\label{ee}
 Two edges $e$ and $d$ of   $\Gamma$   are equivalent   if there is a sequence of edges $e=e_1, e_2,\dots,e_n=d$
of $\Gamma$ with the same initial vertex   such that the terminal
vertices of   $e_i$ and $e_{i+1}$ have a common root for all $i= 1,
\ldots, n-1$.
\end{definition}

 We say that $\Gamma$ has
 property {\rm (EE)} if any   edges of $\Gamma$ with common initial
 vertex are equivalent. The following theorem   is a version of the classical Diamond Lemma
due to Newman~\cite{ne}.

 \begin{theorem}  {(\rm \cite{Mat-HA})}  \label{abstract} If   $\Gamma$
 has
 properties {\rm (F)} and  {\rm (EE)}, then every vertex of $\Gamma$ has a
 unique root. \endproof
 \end{theorem}

 Note that in \rm \cite{Mat-HA} the role of the property   {\rm (F)}  is played by a property {\rm (CF)} which says that
 there is a map $c\colon
 \mathbb{V}
(\Gamma)\to  \{0,1,2,\dots\}$  such that $c(V)>c(W)$ for every edge
$\overrightarrow{VW}$ of $\Gamma$. The property {\rm (F)} implies
{\rm (CF)}; an appropriate   map $c $ is defined as follows: for any
vertex  $V$ of $\Gamma$,  $c(V)$ is the maximal number of edges in a
path in $\Gamma$ starting at $V$.

Recall from Section~\ref{spr} the set   ${\cal U} $ whose elements
are (homeomorphism classes of) theta-curves, labeled knots, and
3-manifolds. We construct an oriented
 graph $\Gamma$ as follows. A vertex of
 $\Gamma $ is a finite sequence of elements of ${\cal U}$
(possibly with repetitions)  considered   up to the following
transformations:  (i)   permutations that change the position of
labeled knots and 3-manifolds in the sequence but keep the order of
theta-curves; (ii)  permutations of two consecutive terms of a
sequence $\theta', \theta'' $ allowed when both terms $\theta',
\theta'' $ are theta-curves and at least one of them is knot-like;
(iii) insertion or deletion trivial theta-curves, trivial labeled
knots, and copies of $S^3$.

 We now define the
 edges of $\Gamma $.
 Let  a vertex $V$ of $\Gamma$ be represented by a sequence $u_1,..., u_n \in \cal U$ and let  $i\in \{1, \ldots, n\}$. Suppose
 that  $(M_1,G_1), (M_2,G_2) $ are obtained from $u_i=(M,G)$ by
an   essential   spherical reduction along a sphere $S\subset M$. If
$G$ is a theta-curve, we choose the numeration so that $(M_1,G_1)$
contains the leg of $ G $ and $(M_2,G_2)$
    contains the head of $G$. If $G$ is a knot or an empty set, then
    the numeration is arbitrary.
 Let $W$ be the vertex of $\Gamma$ represented by the sequence $u_1, \ldots, u_{i-1}, (M_1,G_1), (M_2,G_2), u_{i+1}, \ldots, u_n $.
 We   say that  $W$   is obtained from $V$
by essential spherical reduction along $S$. Two vertices $V,W$ of $
 \Gamma $ are joined by an edge $\overrightarrow{VW}$ if
$W$ can be obtained from $V$ in this way.

\begin{lemma} \label{bound} $\Gamma$
has property (F).
\end{lemma}
\begin{proof} This is a special case of Lemma 6 of~\cite{Mat-HA}.
\end{proof}

\begin{definition}\label{branch} For each   $u \in \cal U$, we define a
  subgraph $\Gamma_u$ of $\Gamma$ as
follows.    The   vertices of $ \Gamma_u $ are all vertices of
$\Gamma$ subordinate to the  vertex of $\Gamma$ represented by the
1-term sequence   $u$. The edges of $\Gamma_u$ are all the edges of
$\Gamma$ with both endpoints in $ \Gamma_u$.
\end{definition}

\begin{lemma}\label{ee} Let $u=(M ,G )$ be a theta-curve or a labeled knot such that
 all $2$-spheres in $M$ are separating.  Then $\Gamma_u$ has property  (EE).
 \end{lemma}
 \begin{proof} Let $V=(u_1,\ldots, u_n)$ be a  vertex of $\Gamma_u$.
   Suppose that edges $\overrightarrow{VW_1},
 \overrightarrow{VW_2}$ of  $\Gamma_u $ correspond to reductions along
 essential spheres $S_1,S_2$. These spheres lie in the ambient
 3-manifolds of $u_p, u_q$ for some $p, q\in \{1,\ldots ,
 n\}$. If $p\neq q$, then
 $S_1,S_2$
   survive  the reduction along each other.  Thus we may consider $S_1$ as a sphere in (a term of) $W_2$ and
$S_2$ as a sphere in (a term of) $W_1$. Both  these spheres are
essential and
 the reductions of
 $W_2$ along $S_1$ and of $W_1$ along $S_2$ yield the same vertex, $W $, of
 $\Gamma_{u}$.
 Any root   of $W $ is a
 common root of  $ W_1 $  and $ W_2$, and therefore the edges $\overrightarrow{VW_1},
 \overrightarrow{VW_2}$   are equivalent.

 It remains to consider  the case where both spheres
 $S_1,S_2$ lie   in the ambient
 3-manifold $M_p$ of the same term $u_p=(M_p,G_p)$ of $V$. Note that $M_p$
 is a submanifold of $M$ and therefore all $2$-spheres in $M_p$ are
 separating.
We prove the equivalence of the  edges $\overrightarrow{VW_1},
 \overrightarrow{VW_2}$ by induction on the
 number $m$
 of circles in   $S_1\cap S_2$.

  {\em Base of induction}. Let $m=0$, i.e., $S_1,S_2$ are disjoint.  Then each of
  these spheres
   survives the reduction along the other.  Thus we may consider $S_1$ as a sphere in (a term of) $W_2$ and
$S_2$ as a sphere in (a term of) $W_1$. Consider   the  vertices
$W_3, W_3'$ of $\Gamma_{u}$
       obtained by reducing
 $W_2$ along $S_1$ and $W_1$ along $S_2$. Let us  prove  that $W_3=W'_3$,
 see the diagram on the left-hand side of Figure~\ref{vista}. Assume first that each sphere $S_1,S_2$ meets $G_p$ in three points.
Then $G_p$ is a
 theta-curve and the reductions of $(M_p,G_p)$ along $S_1,S_2$ give
 three nontrivial theta-curves $\theta_i=(Q_i,\Theta_i), 1\leq i\leq 3$,
 where $ \Theta_1 $ and $ \Theta_3 $  contain the leg and  the
 head of $G_p$, respectively.  It follows that
 both
 $W_3$ and $ W_3'$ are obtained from $V$ by replacing the term $(M_p, G_p)$ with 3
 terms  $\theta_1, \theta_2,\theta_3$, see Figure~\ref{vista}.
 The other cases where at least one of the spheres $S_1,S_2$ meets $G_p$ in 2 or 0
 points
  are treated
 similarly.

\begin{figure}
\centerline{\psfig{figure=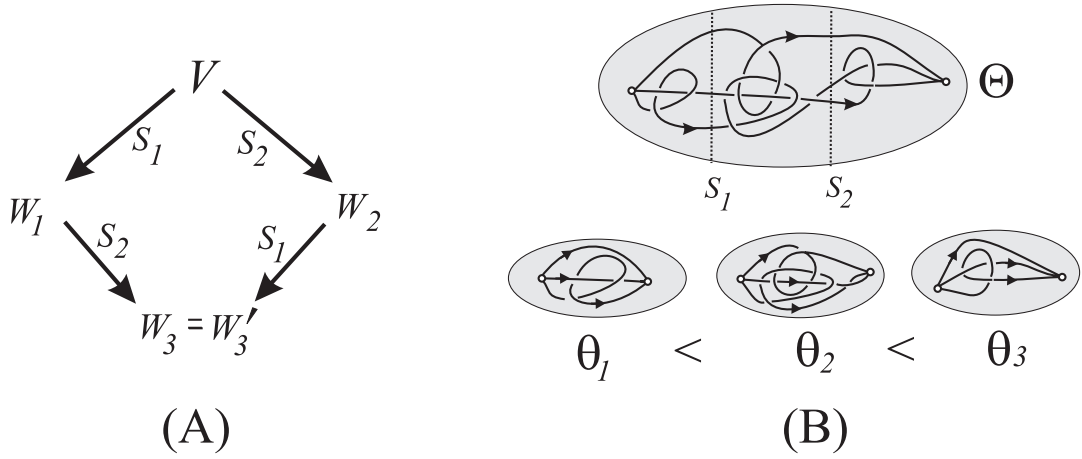,height=6cm}}
  \caption{(A) Reductions along disjoint spheres. (B) The ordering of $\theta_i$ is natural }
  \label{vista}
\end{figure}

    We claim that any root $R$ of $W_3$ is a
 common root of  $ W_1 $  and $ W_2$.
  Indeed,   if   $S_1$ is  essential in $W_2$, then $R$ is a root
  of $W_2$ by the definition of a root. If $S_1$ is inessential in
$W_2$, the  reduction  along it results in adding  to $W_2$ either
$S^3$, or a trivial knot, or a trivial theta-curve. Then $W_2=W_3$
by the definition of a vertex of $\Gamma$. Therefore $R$ is a root
of $W_2$. Similarly, $R$ is a root of $W_1$. Therefore,
 the edges
 $\overrightarrow{VW_1},
 \overrightarrow{VW_2}$  are
  equivalent.

{\em Inductive step.} Let $\# (S_1\cap
 S_2)=m+1$. It follows from Lemma~\ref{clean} that   there is an essential
 sphere-mediator $S_3$ such that it intersects $S_1$ and $S_2$ in
 a smaller number of circles. By the inductive assumption we know
 that the corresponding edge $\overrightarrow{VW_3}$ is equivalent to $\overrightarrow{VW_1}$ and
 $\overrightarrow{VW_2}$. It follows that $\overrightarrow{VW_1}$
 and $\overrightarrow{VW_2}$ are also equivalent.
\end{proof}

\begin{corollary} \label{mainc} Any vertex of $\Gamma_{u}$
has a unique root.
\end{corollary}

This   follows from  Theorem~\ref{abstract}  and Lemmas~\ref{bound}
and ~\ref{ee}.

\section{Proof of  Theorem~\ref{main}}        We apply to
$\theta$ consecutive reductions along essential spheres   meeting
the corresponding theta-curves in three points. After $m\geq 1$
reductions we obtain  a sequence of $m+1$ theta-curves. Since
$\Gamma$ has property (F), for some $m$ there will be no
essential spheres meeting the corresponding theta-curves in three
points. This means that all the theta-curves obtained after $m $
reductions are prime.  We obtain thus a sequence of prime
theta-curves
 whose product is equal to $\theta$. This proves the first claim of the theorem.

We now prove the second claim. Consider  an expansion of $\theta$
as a product of $n$ prime theta-curves $ \theta_1, \ldots,
\theta_n $. Let $W$ be the sequence $ \theta_1, \ldots, \theta_n
$.  It may happen that the theta-curve $\theta_j =(Q_j ,\Theta_j
)$ admits an essential reduction along a sphere $S\subset Q_j$
meeting $\Theta_j$ in two points. These points have to lie on the
same edge $e$ of $\Theta_j$ because otherwise the sphere $S$ would
be non-separating. If $i\in \{-,0,+\}$ is the label of $e$, then
this spherical reduction produces  a theta-curve $\theta'_j$ and a
knot $k_j\in {\cal K}_i$ such that $\theta_j=\theta'_j \circ
\tau_i(k_j)$. Since $\theta_j$ is prime, $\theta'_j$ is trivial.
We may conclude that $\theta_j=\tau_i(k_j)$ is knot-like, where
$k$ is a prime knot by Lemma~\ref{le2}. Similarly, if $\theta_j
=(Q_j ,\Theta_j )$ admits an essential reduction along a sphere
disjoint from $\Theta_j$, then  $\theta_j=\tau (Q)$ is also
knot-like, where $Q$ is a prime manifold.

 Replacing in the sequence $W$ all knot-like $\theta_j$ by the corresponding
knots $k_j$, we obtain a root of the vertex $\theta$ of $\Gamma$. By
the uniqueness of the root (Corollary~\ref{mainc}),  the expansion
$\theta=\prod_{j=1}^n\theta_j$ is unique up to the commutation
relations of knot-like theta-curves with all the others.


 \section{Corollaries}

\begin{theorem} \label{mainforknots} Let $k=(Q,K)$ be a non-trivial knot
such that all $2$-spheres in $Q$ are separating. Then  $k$ expands
as a connected sum $k= k_1\# k_2\# \dots \# k_n$ of $ n\geq 1$ prime
knots. This expansion is unique up to permutations of $k_1,  ...,
k_n$.
\end{theorem}
\begin{proof}   Pick any $i\in \{-, 0, +\}$. The claim follows from Theorem~\ref{main}, the injectivity of the semigroup homomorphism
  $  \tau_i\colon {\cal K} \to
{\cal T} $, and the fact that $k\in {\cal K}$ is prime if and only
if $\tau_i(k)$ is prime
 (Lemma~\ref{le2}). \end{proof}

 A more general version of this theorem was proved
by Miyazaki~\cite{miya}.

Let ${\cal U}^{\circ}={\cal T}^{\circ}  \bigsqcup   {\cal
K}^{\circ}_- \bigsqcup {\cal K}^{\circ}_0 \bigsqcup   {\cal
K}^{\circ}_+ \bigsqcup {\cal
  M}^{\circ}$
be a subset of ${\cal U}$ consisting of the pairs $(M,G)$ such
  that all spheres in  $M$ are separating. Any element of ${\cal U}^{\circ}$
  expands as a product (or connected sum) of prime elements of  ${\cal U}^{\circ}$. Let
${\widehat {\cal T}}^{\circ}$ be the subsemigroup of ${\cal
T}^{\circ}$ consisting of theta-curves having no knot-like
factors. Similarly, denote by ${\widehat {\cal K}}^{\circ}$ the
subsemigroup of ${\cal K}^{\circ}$ consisting of knots having no
3-manifold summands. A knot $(Q,K)\in {\cal K}^{\circ}$ lies in
${\widehat {\cal K}}^{\circ}$  if and only if $Q\setminus K$ is an
irreducible $3$-manifold. For $i\in \{-,0,+ \}$ we denote by
${\widehat {\cal K}}^{\circ}_i$ a copy of ${\widehat {\cal
K}}^{\circ}$ formed by $i$-labeled knots.

\begin{theorem} \label{cor} The following holds:

\begin{enumerate}

\item $ {\cal M}^{\circ}$ and $ \widehat {\cal K}^0$ are free
abelian semigroups freely generated by their prime elements.

\item  ${\widehat {\cal T}}^{\circ}$ is a free   semigroup  freely generated by its
prime elements.

\item $ {\cal K}^{\circ}={\widehat {\cal K}}^{\circ} \times {\cal M}^{\circ}$.

\item ${\cal T}^{\circ}= {\widehat {\cal T}}^{\circ} \times {\cal C}$,   where  ${\cal
C}= {\widehat {\cal K}}^{\circ}_- \times {\widehat {\cal
K}}^{\circ}_0 \times {\widehat {\cal K}}^{\circ}_+ \times  {\cal
M}^{\circ}$  is the center of ${\cal T}^{\circ}$.
\end{enumerate}
\end{theorem}

\begin{proof} This follows from  Theorems~\ref{main} and~\ref{mainforknots}.
\end{proof}

Sergei Matveev

Chelyabinsk State University

Kashirin Brothers Str. 129

Chelyabinsk  454001

Russia

e-mail: matveev@csu.ru

\vskip1truecm

Vladimir Turaev

Department of Mathematics

 \indent Indiana University

                     \indent Bloomington IN47405

                     \indent USA

\indent e-mail: vtouraev@indiana.edu


\begin{thebibliography}{9999999999}
\bibitem[Fox]{F} R. Fox,  A quick trip through knot theory, in: M.K. Fort (Ed.),
"Topology of $3$-Manifolds and Related Topics", Prentice-Hall, NJ,
1961, pp. 120–167.

\bibitem[HM]{Mat-HA} C. Hog-Angeloni, S. Matveev, Roots in $3$-manifold topology,
Geometry \& Topology Monographs 14 (2008), 295-319.

\bibitem[Kn]{Kneser}   H. Kneser, Geschlossene Fl\"achen in
dreidimensionalen Mannigfaltigkeiten. Jahresbericht der Deut.
Math. Verein, 28 (1929), 248-260.

\bibitem[Mi]{Milnor}  J. Milnor, A unique factorisation theorem for $3$-manifolds,
Amer. J. Math.  84 (1962),  1-7.

\bibitem[Miy]{miya} K. Miyazaki, Conjugation and prime
decomposition of knots in closed, oriented $3$-manifolds,
Transactions of the American Mathematical Society,   313
(1989), 785-804.

\bibitem[Mo]{Mo1}  T.   Motohashi,   A prime decomposition theorem for
 handcuff graphs in $S^3$.  Topology Appl.  154  (2007),   3135--3139.

\bibitem[Ne]{ne} M. H. A. Newman, On theories with a
combinatorial definition of \lq\lq equivalence", Ann. Math. 43 (1942),
223-243.



\bibitem[Pe]{Petronio} C. Petronio, Spherical splitting of $3$-orbifolds,
 Math. Proc. Cambridge
Philos. Soc. 142 (2007), 269-287.


\bibitem[Sh]{Shubert} H. Shubert,  Die eindeutige Zerlegbarkeit
eines Knotens in Primknoten, S.B.-Heidelberger Akad. Wiss. Math.
Natur. Kl. 3 (1949), 57-104.

\bibitem[Tu]{tu2} V. Turaev, Knotoids, arXiv:1002.4133.


\bibitem[Wo]{wo} K. Wolcott,  The knotting of theta curves and other graphs in
$S^3$.  Geometry and topology (Athens, Ga., 1985),  325--346,
Lecture Notes in Pure and Appl. Math., 105, Dekker, New York,
1987.

\end{thebibliography}
\end{document}